\documentclass[12pt]{article}
\usepackage[english]{babel}
\usepackage[all]{xy}
\usepackage{fancyhdr}
\usepackage{setspace, amsmath, amsthm, amssymb, amsfonts, amscd, epic, graphicx, ulem, dsfont}
\usepackage{amsthm}
\usepackage[T1]{fontenc}
\usepackage{cases}
\newtheorem{theorem}{Theorem}

\newtheorem{corollary}[theorem]{Corollary}
\newtheorem{example}[theorem]{Example}

\newtheorem{remark}[theorem]{Remark}

\makeatletter

\@addtoreset{equation}{section}
\makeatother

\title{On the generalized of $p$-biharmonic and bi-$p$-harmonic maps}

\author{Fethi Latti \footnote{Salhi Ahmed Naama University Center, Department of mathematics, Naama 45000, Algeria. Email: 
etafati@hotmail.fr
} and Ahmed Mohammed Cherif\footnote{University Mustapha Stambouli Mascara, Faculty of Exact Sciences, Mascara 29000, Algeria. Email: med\_cherif\_ahmed@yahoo.fr}}
\date{}

\begin{document}
\maketitle
	
\begin{abstract}
In this note, we extend the definition of $p$-biharmonic and bi-$p$-harmonic
maps between two Riemannian manifolds and explore some of their properties.
\begin{flushleft}
\end{flushleft}
\end{abstract}

\section{Introduction}
$p$-Harmonic maps $\varphi:(M,g)\to(N,h)$ between Riemannian manifolds are the critical points of the $p$-energy functional
\begin{equation}\label{eq1.1A}
E_p(\varphi;D)=\frac{1}{p}\int_{D} |d\varphi|^pv_{g},
\end{equation}
for any compact domain $D\subset M$ for some constant $p\geq2$. The corresponding Euler-Lagrange equation of (\ref{eq1.1A}) is giving by the vanishing of the $p$-tension field
\begin{equation}\label{eq1.1B}
\tau_{p}(\varphi)=\operatorname{div}(|d\varphi|^{p-2}d\varphi).
\end{equation}
A smooth map $\varphi$ is $p$-harmonic if and only if
$\tau_{p}(\varphi)=0$ (for more details on the concept of $p$-harmonic maps see \cite{BG,BI,ali}). The $p$-harmonic maps are a natural generalization of the harmonic maps (see \cite{baird, ES}).\\
The study of $p$-biharmonic maps extends the classical theory of harmonic and biharmonic maps  (see \cite{Jiang, YX}).
A $p$-biharmonic map $\varphi:(M,g)\to(N,h)$ between Riemannian manifolds is a critical point of the $p$-biharmonic energy functional (see \cite{CL2,HL}) defined by
\begin{equation}\label{eq1.1}
	E_{2,p}(\varphi;D)=\frac{1}{p}\,\int_{D}|\tau(\varphi)|^{p}\,v_{g},
\end{equation}
where $\tau(\varphi)$ is the tension field of $\varphi$ giving by $\tau(\varphi)=\operatorname{trace}\nabla d\varphi$.
In \cite{cherif2}, the author introduced the notion of bi-$p$-harmonic maps as follows.
Consider a smooth map $\varphi:(M,g)\rightarrow (N,h)$, the bi-$p$-energy of $\varphi$ is defined by
	\begin{equation}\label{eq1.2}
	E_{p,2}(\varphi;D)=\frac{1}{2}\,\int_{D}|\tau_{p}(\varphi)|^{2}\,v_{g}.
	\end{equation}
The map $\varphi$ is called bi-$p$-harmonic if it is a critical point of the bi-$p$-energy functional (\ref{eq1.2}) for any compact domain $D\subset M$.\\
In this paper, we further extend the definition of $p$-biharmonic and bi-$p$-harmonic maps by introducing
$(p,q)$-harmonic maps. A map $\varphi:(M,g)\to(N,h)$ between Riemannian manifolds is called $(p,q)$-harmonic if it is a critical point of the $(p,q)$-energy functional
\begin{equation}\label{eq1.3}
E_{p,q}(\varphi;D)=\frac{1}{q}\int_{D} |\tau_{p}(\varphi)|^qv_{g},
\end{equation}
for any compact domain $D\subset M$ for some constants $p,q\geq2$. Note that $(2,2)$-harmonic maps are biharmonic maps.
In \cite{cherif1,cherif3}, the author studies Liouville-type theorems for generalized $p$-harmonic maps.
In \cite{Ara}, the author studies the geometry of \( F\)-harmonic maps, where \( F \in C^\infty(\mathbb{R}) \) is a positive function. This study extends the framework of \( p \)-harmonic maps by considering higher-order energy functionals influenced by the function \( F \).
We compute the first variation of the $(p,q)$-energy functional, and establish Liouville-type theorems for $(p,q)$-harmonic maps. Additionally, we construct new examples of proper $(p,q)$-harmonic maps, that is $(p,q)$-harmonic non $p$-harmonic maps. These concepts generalize various energy functionals in geometric analysis and have applications in the study of nonlinear differential equations.


\section{First variation of the $(p,q)$-energy functional}

The following theorem provides an explicit expression for the first variation of \( E_{p,q} \), and introduces the associated \((p,q)\)-tension field \( \tau_{p,q}(\varphi) \), whose vanishing characterizes critical points.

\begin{theorem}\label{th01}
  Let $\varphi$ be a smooth map from Riemannian manifold $(M,g)$ to Riemannian manifold $(N,h)$, and $\{\varphi_{t}\}_{t\in(-\epsilon,\epsilon)}$ a smooth variation of $\varphi$ to support in compact domain $D\subset M$. Then
\begin{equation}\label{equa01}
  \frac{d}{dt}E_{p,q}(\varphi_{t};D)\Big|_{t=0}=-\int_{D}h\big(v,\tau_{p,q}(\varphi)\big)v_{g},
\end{equation}
where $\tau_{p,q}(\varphi)$ is the $(p,q)$-tension field of $\varphi$ given by
\begin{eqnarray*}
   \tau_{p,q}(\varphi)&=&-|d\varphi|^{p-2}|\tau_{p}(\varphi)|^{q-2}\operatorname{trace}R^{N}\big(\tau_{p}(\varphi),d\varphi\big)d\varphi\\
   &&- \operatorname{trace}\nabla^{\varphi}|d\varphi|^{p-2} \nabla^{\varphi}|\tau_{p}(\varphi)|^{q-2}\tau_{p}(\varphi)\\
   &&-(p-2)\operatorname{trace}\nabla^{\varphi}|d\varphi|^{p-4}
   \left\langle\nabla^{\varphi}|\tau_{p}(\varphi)|^{q-2}\tau_{p}(\varphi),d\varphi\right\rangle d\varphi,
\end{eqnarray*}
and $v=\frac{d\varphi_{t}}{dt}|_{t=0}$ denotes the variation vector field of $\{\varphi_{t}\}_{t\in(-\epsilon,\epsilon)}$.
\end{theorem}

\begin{proof}
Let $\phi$ be a smooth map defined by
\begin{eqnarray*}
   \phi:(-\epsilon,\epsilon)\times M&\rightarrow& N. \\
  (t,x)\quad &\mapsto&\phi(t,x)=\varphi_{t}(x)
\end{eqnarray*}
Let $\{e_{i}\}$ be an orthonormal frame with respect to $g$ on $M$ such that $\nabla_{e_{j}}^{M}e_{j}=0$ at $x\in M$ for all $i,j=1, ...,m$.\\ Note that $\phi(x,0)=\varphi(x)$, and the variation vector field $v$ associated to the variation $\{\varphi_{t}\}_{t\in(-\epsilon,\epsilon)}$ is given by $v=d\phi(\frac{\partial}{\partial t})|_{t=0}$. We have
\begin{equation}\label{equa005}
  \frac{d}{dt}E_{p,q}(\varphi_{t};D)=\frac{1}{q}\int_{D}\frac{\partial}{\partial t}|\tau_{p}(\varphi)|^{q}v^{g}.
\end{equation}
We compute the following term
\begin{eqnarray}\label{eq2.3}
  \frac{\partial}{\partial t}|\tau_{p}(\varphi_{t})|^{q}
   &=&\nonumber \frac{\partial}{ \partial t}\left[|\tau_{p}(\varphi_{t})|^{2}\right]^{\frac{q}{2}} \\
   &=&\nonumber \frac{\partial}{ \partial t}\left[h(\tau_{p}(\varphi_{t}),\tau_{p}(\varphi_{t}))\right]^{\frac{q}{2}} \\
   &=&\nonumber \frac{q}{2}h(\tau_{p}(\varphi_{t}),\tau_{p}(\varphi_{t}))^{\frac{q}{2}-1}\frac{\partial}{\partial t}h(\tau_{p}(\varphi_{t}),\tau_{p}(\varphi_{t})) \\
   &=&q h(\tau_{p}(\varphi_{t}),\tau_{p}(\varphi_{t}))^{\frac{q-2}{2}}h( \nabla_{\frac{\partial}{\partial t}}^{\phi} \tau_{p}(\varphi_{t}),\tau_{p}(\varphi_{t})).
\end{eqnarray}
Take $E_i=(0,e_i)$ for all $i=1,...,m$. As $\nabla_{E_{i}}E_{i}=0$ at $x$, we get
\begin{eqnarray*}
  \nabla_{\frac{\partial}{\partial t}}^{\phi} \tau_{p}(\varphi_{t})
   &=& \nabla_{\frac{\partial}{\partial t}}^{\phi}\Big[\nabla_{E_{i}}^{\phi}|d\varphi_{t}|^{p-2}d\phi(E_{i})
       -|d\varphi_{t}|^{p-2}d\phi(\nabla_{E_{i}}E_{i})\Big]  \\
   &=& \nabla_{\frac{\partial}{\partial t}}^{\phi}\nabla_{E_{i}}^{\phi}|d\varphi_{t}|^{p-2}d\phi(E_{i}).
\end{eqnarray*}
From the definition of the curvature tensor of $(N,h)$, we obtain
\begin{eqnarray}\label{eq2.4}
  \nabla_{\frac{\partial}{\partial t}}^{\phi} \tau_{p}(\varphi_{t})
   &=&\nonumber R^{N}\big(d\phi(\frac{\partial}{\partial t}),d\phi(E_{i})\big)|d\varphi_{t}|^{p-2}d\phi(E_{i})
      +\nabla_{E_{i}}^{\phi}\nabla_{\frac{\partial}{\partial t}}^{\phi}|d\varphi_{t}|^{p-2}d\phi(E_{i}) \\
   && + \nabla^\phi_{[\frac{\partial}{\partial t},E_{i}]}|d\varphi_{t}|^{p-2}d\phi(E_{i})\\
   &=&\nonumber |d\varphi_{t}|^{p-2} R^{N}\big(d\phi(\frac{\partial}{\partial t}),d\phi(E_{i})\big)d\phi(E_{i})
      +\nabla_{E_{i}}^{\phi}\nabla_{\frac{\partial}{\partial t}}^{\phi}|d\varphi_{t}|^{p-2}d\phi(E_{i}).
\end{eqnarray}
By substituting (\ref{eq2.4}) in (\ref{eq2.3}), we find that
\begin{eqnarray}\label{equa03}
   \frac{1}{q}\frac{\partial}{\partial t}|\tau_{p}(\varphi_{t})|^{q}
   &=& |d\varphi_{t}|^{p-2} |\tau_{p}(\varphi_{t})|^{q-2}h\big(R^{N}\big(d\phi(\frac{\partial}{\partial t}),d\phi(E_{i})\big)d\phi(E_{i}),\tau_{p}(\varphi_{t}) \big)\nonumber \\
   && + E_{i}\Big[|\tau_{p}(\varphi_{t})|^{q-2}h(\nabla_{\frac{\partial}{\partial t}}^{\phi}|d\varphi_{t}|^{p-2}d\phi(E_{i}),\tau_{p}(\varphi_{t}))\Big]\nonumber\\
   &&-h(\nabla_{\frac{\partial}{\partial t}}^{\phi}|d\varphi_{t}|^{p-2}d\phi(E_{i}),\nabla_{E_{i}}^{\phi}|\tau_{p}(\varphi_{t})|^{q-2}\tau_{p}(\varphi_{t})).
\end{eqnarray}
By using the property $\nabla^{\phi}_{X}d\phi(Y)=\nabla^{\phi}_{Y}d\phi(X)+d\phi([X,Y])$, we get
\begin{eqnarray}\label{eq2.6}
\nabla_{\frac{\partial}{\partial t}}^{\phi}|d\varphi_t|^{p-2}d\phi(E_{i})\Big|_{t=0}
&=& \nabla_{\frac{\partial}{\partial t}}^{\phi}d\phi(|d\varphi_t|^{p-2}E_{i})\Big|_{t=0}\\
&=&\nonumber |d\varphi|^{p-2}\nabla^\varphi_{e_{i}}v+\frac{\partial}{\partial t}|d\varphi_t|^{p-2}\Big|_{t=0}d\varphi(e_i)\\
&=&\nonumber |d\varphi|^{p-2}\nabla^\varphi_{e_{i}}v
+(p-2)|d\varphi|^{p-4}h(\nabla^\varphi_{e_{j}}v,d\varphi(e_{j}))d\varphi(e_{i}).
\end{eqnarray}
By substituting (\ref{eq2.6}) in (\ref{equa03}), we obtain
\begin{eqnarray}\label{eq2.7}
   \frac{1}{q}\frac{\partial}{\partial t}|\tau_{p}(\varphi_{t})|^{q}\Big|_{t=0}
   &=&\nonumber |d\varphi|^{p-2} |\tau_{p}(\varphi)|^{q-2}h\big(R^{N}\big(v,d\varphi(e_{i})\big)d\varphi(e_{i}),\tau_{p}(\varphi) \big) \\
   &&\nonumber + e_{i}\Big[|d\varphi|^{p-2}|\tau_{p}(\varphi)|^{q-2}h(\nabla^\varphi_{e_{i}}v,\tau_{p}(\varphi))\Big]\\
   &&\nonumber + (p-2) e_{i}\Big[|d\varphi|^{p-4}|\tau_{p}(\varphi)|^{q-2}h(\nabla^\varphi_{e_{j}}v,d\varphi(e_{j}))h(d\varphi(e_{i}),\tau_{p}(\varphi))\Big]\nonumber\\
   &&\nonumber-|d\varphi|^{p-2}h(\nabla^\varphi_{e_{i}}v,\nabla_{e_{i}}^{\varphi}|\tau_{p}(\varphi)|^{q-2}\tau_{p}(\varphi_{t}))\\
   &&\nonumber-(p-2)|d\varphi|^{p-4}h(\nabla^\varphi_{e_{j}}v,d\varphi(e_{j}))
   h(d\varphi(e_{i}),\nabla_{e_{i}}^{\varphi}|\tau_{p}(\varphi)|^{q-2}\tau_{p}(\varphi)).
\end{eqnarray}
Let $\eta_{1},\eta_{2},\eta_{3},\eta_{4}\in \Gamma(T^{*}M)$ defined by
\begin{eqnarray*}
\eta_{1}(X)&=&|d\varphi|^{p-2}|\tau_p(\varphi)|^{q-2}h\big(\nabla_{X}^{\varphi}v,\tau_{p}(\varphi) \big),\\
\eta_{2}(X)&=&|d\varphi|^{p-4}|\tau_p(\varphi)|^{q-2}\left\langle\nabla^{\varphi}v,d\varphi\right\rangle h\big(d\varphi(X),\tau_{p}(\varphi)\big),\\
\eta_{3}(X)&=&|d\varphi|^{p-2}h\big(v,\nabla_{X}^{\varphi}|\tau_{p}(\varphi)|^{q-2}\tau_{p}(\varphi) \big),\\
\eta_{4}(X)&=&|d\varphi|^{p-4}
\left\langle\nabla^{\varphi}|\tau_{p}(\varphi)|^{q-2}\tau_{p}(\varphi),d\varphi\right\rangle h\big(v,d\varphi(X)\big).
\end{eqnarray*}
Therefore (\ref{eq2.7}) becomes
\begin{eqnarray*}
   \frac{1}{q}\frac{\partial}{\partial t}|\tau_{p}(\varphi_{t})|^{q}\Big|_{t=0}
   &=&|d\varphi|^{p-2}|\tau_{p}(\varphi)|^{q-2}h\big(R^{N}\big(\tau_{p}(\varphi),d\varphi(e_{i}) \big)d\varphi(e_{i}),v \big)\\
   &&+ \operatorname{div}\eta_{1}+(p-2)\operatorname{div}\eta_{2}-\operatorname{div}\eta_{3}-(p-2)\operatorname{div}\eta_{4} \\
   &&+ h\big(v,\nabla_{e_{i}}^{\varphi}|d\varphi|^{p-2}\nabla_{e_{i}}^{\varphi}|\tau_{p}(\varphi)|^{q-2}\tau_{p}(\varphi)\big)  \\
   && +(p-2)h\big(v,\nabla_{e_{j}}^{\varphi}|d\varphi|^{p-4}
   \left\langle\nabla^{\varphi}|\tau_{p}(\varphi)|^{q-2}\tau_{p}(\varphi),d\varphi\right\rangle d\varphi(e_{j}) \big).
\end{eqnarray*}
The Theorem \ref{th01} follows from the divergence Theorem.
\end{proof}
According to Theorem \ref{th01}, we deduce the following Corollaries.

\begin{corollary}
  A smooth map $\varphi:(M,g)\rightarrow(N,h)$ between two Riemannian manifolds is $(p,q)$-harmonic if and only if
\begin{eqnarray*}
   \tau_{p,q}(\varphi)&=&-|d\varphi|^{p-2}|\tau_{p}(\varphi)|^{q-2}\operatorname{trace}R^{N}\big(\tau_{p}(\varphi),d\varphi\big)d\varphi\\
   &&- \operatorname{trace}\nabla^{\varphi}|d\varphi|^{p-2} \nabla^{\varphi}|\tau_{p}(\varphi)|^{q-2}\tau_{p}(\varphi)\\
   &&-(p-2)\operatorname{trace}\nabla^{\varphi}|d\varphi|^{p-4}
   \left\langle\nabla^{\varphi}|\tau_{p}(\varphi)|^{q-2}\tau_{p}(\varphi),d\varphi\right\rangle d\varphi=0.
\end{eqnarray*}
\end{corollary}

\begin{corollary}\cite{cherif2}
  A smooth map $\varphi:(M,g)\rightarrow(N,h)$ between two Riemannian manifolds is bi-$p$-harmonic if and only if
  \begin{eqnarray*}
   \tau_{p,2}(\varphi)&=&-|d\varphi|^{p-2}\operatorname{trace}R^{N}\big(\tau_{p}(\varphi),d\varphi\big)d\varphi\\
   &&- \operatorname{trace}\nabla^{\varphi}|d\varphi|^{p-2} \nabla^{\varphi}\tau_{p}(\varphi)\\
   &&-(p-2)\operatorname{trace}\nabla^{\varphi}|d\varphi|^{p-4}
   \left\langle\nabla^{\varphi}\tau_{p}(\varphi),d\varphi\right\rangle d\varphi=0.
\end{eqnarray*}
\end{corollary}

\begin{corollary}\cite{CL2,HL}
  A smooth map $\varphi:(M,g)\rightarrow(N,h)$ between two Riemannian manifolds is $p$-biharmonic if and only if
  \begin{eqnarray*}
   \tau_{2,p}(\varphi)
   &=&-|\tau(\varphi)|^{p-2}\operatorname{trace}R^{N}\big(\tau(\varphi),d\varphi\big)d\varphi\\
    &&  - \operatorname{trace}(\nabla^{\varphi})^2|\tau(\varphi)|^{p-2}\tau(\varphi)=0.
\end{eqnarray*}
\end{corollary}

\begin{remark}
 $(1)$ Any $p$-harmonic map is $(p,q)$-harmonic maps, i.e., $\tau_{p}(\varphi)=0$ implies $\tau_{p,q}(\varphi)=0$.\\
 $(2)$ Let $\varphi:(M,g)\rightarrow(N,h)$ be a smooth map. If the pull-back vector field $|\tau_{p}(\varphi)|^{q-2}\tau_{p}(\varphi)$ is parallel along $\varphi$ (in the case $N=\mathbb{R}^n$ we assume that the components of $|\tau_{p}(\varphi)|^{q-2}\tau_{p}(\varphi)$ are constants), then the map $\varphi$ is $(p,q)$-harmonic.
 Note that we have construction many examples of proper $(p,q)$-harmonic, that is $(p,q)$-harmonic non $p$-harmonic map. \\ $(3)$ Every bi-$p$-harmonic map $\varphi:(M,g)\to (N,h)$ with parallel $p$-tension field $\tau_{p}(\varphi)$ is $(p,q)$-harmonic.
\end{remark}

To deepen the understanding of \((p,q)\)-harmonic maps, it is of particular interest to examine explicit examples that satisfy the corresponding Euler--Lagrange equations. In certain cases, bi-\(p\)-harmonic maps may also satisfy the \((p,q)\)-harmonic condition. The following example illustrates such a case and highlights the interplay between these two notions.

\begin{example}
Let $M$ the manifold $\mathbb{R}^2\backslash\{(0,0)\}\times\mathbb{R}$ equipped with the Riemannian
metric $g=(x^2+y^2)^{-\frac{1}{p}}(dx^2+dy^2+dz^2)$ and let $N=\mathbb{R}^2$ equipped with the Riemannian
metric $h=du^2+dv^2$. The bi-$p$-harmonic map $\varphi:(M,g)\to \mathbb{R}^2$ defined by
$\varphi(x,y,z)=(\sqrt{x^2+y^2},z)$ (see \cite{cherif2, MM2}) is also $(p,q)$-harmonic for all $p,q\geq2$, because the components of $|\tau_{p}(\varphi)|^{q-2}\tau_{p}(\varphi)$ are constants. A straightforward computation yields
$$|\tau_{p}(\varphi)|^{q-2}\tau_{p}(\varphi)=2^{\frac{(p-2)(q-1)}{2}}\left(2-\frac{3}{p}\right)^{q-1}\frac{\partial}{\partial u}.$$
\end{example}

The following corollary establishes a useful equivalence under an assumption. This condition, appears naturally in several geometric contexts and simplifies the characterization of critical points of the \((p,q)\)-energy functional.

\begin{corollary}
  Let $\varphi:(M,g)\rightarrow(N,h)$ be a smooth map between two Riemannian manifolds. We assume that $|\tau_{p}(\varphi)|^{q-2}$ is constant for some $p,q\geq2$. Then, $\varphi$ is $(p,q)$-harmonic if and only if it is
  bi-$p$-harmonic.
\end{corollary}

In the next example, we demonstrate the existence of a proper $(p,q)$-harmonic maps from hyperbolic space.

\begin{example}
The following map from $4$-dimensional hyperbolic space into $4$-dimensional Euclidean space
\begin{eqnarray*}
  \varphi:(\mathbb{H}^4,x_4^{-2/p}(dx_1^2+...+dx_4^2)) &\to& (\mathbb{\mathbb{R}}^4,dy_1^2+...+dy_4^2), \\
   (x_1,...,x_4)&\longmapsto&(x_1,...,x_4)
\end{eqnarray*}
is $(p,q)$-harmonic and bi-$p$-harmonic, where $p>4$ and $q\geq2$. From a simple calculation, we get
\begin{eqnarray*}
  |\tau_{p}(\varphi)|^{q-2} &=& 2^{(p-2)(q-2)} p^{2-q}(p-4)^{q-2},\\
   |\tau_{p}(\varphi)|^{q-2}\tau_{p}(\varphi)&=&2^{(p-2)(q-1)} p^{1-q}(p-4)^{q-1}\frac{\partial}{\partial y_4}.
\end{eqnarray*}
\end{example}

\begin{remark}
There exist non bi-$p$-harmonic non $p$-biharmonic $(p,q)$-harmonic map $\varphi$ such that $|\tau_{p}(\varphi)|^{q-2}\tau_{p}(\varphi)$ is non-parallel along $\varphi$.
\end{remark}

\begin{example}
Let $s\in\mathbb{R}\backslash\{1\}$. For all $p,q\geq2$, the smooth map
\begin{eqnarray*}
   \varphi:(\mathbb{R}^2,dx^2+dy^2)&\to&(\mathbb{R}^2,du^2+dv^2), \\
   (x,y)&\mapsto&(x^{s},0)
\end{eqnarray*}
is $(p,q)$-harmonic if and only if $s=\frac{p}{p-1}$ or $s=\frac{pq-1}{pq-q-1}$.
Indeed; A straightforward computation yields
\begin{eqnarray}\label{eq28}
 |d\varphi|&=&\nonumber sx^{s-1},\\
 \tau_{p}(\varphi)  &=&\nonumber s^{p-1}(ps-p-s+1)x^{ps-p-s}\frac{\partial}{\partial u}, \\
 |\tau_{p}(\varphi)|^{q-2}  &=&  \left[s^{p-1}(ps-p-s+1)x^{ps-p-s}\right]^{q-2},\\
 \left\langle\nabla^{\varphi}|\tau_{p}(\varphi)|^{q-2}\tau_{p}(\varphi),d\varphi\right\rangle
 &=&\nonumber(q-1)(ps-p-s)(ps-p-s+1)^{q-1}\\
 & &\nonumber\quad s^{(p-1)(q-1)+1}x^{(q-1)(ps-p-s)+s-2}.
\end{eqnarray}
Since $R^{\mathbb{R}^2}=0$ and $d\varphi(\frac{\partial}{\partial y})=0$, the $(p,q)$-tension field of $\varphi$ is given by
\begin{eqnarray}\label{eq29}
   \tau_{p,q}(\varphi)
   &=&\nonumber-\nabla^{\varphi}_{\frac{\partial}{\partial x}}|d\varphi|^{p-2} \nabla^{\varphi}_{\frac{\partial}{\partial x}}|\tau_{p}(\varphi)|^{q-2}\tau_{p}(\varphi)\\
    &&-\nabla^{\varphi}_{\frac{\partial}{\partial y}}|d\varphi|^{p-2} \nabla^{\varphi}_{\frac{\partial}{\partial y}}|\tau_{p}(\varphi)|^{q-2}\tau_{p}(\varphi)\\
   &&\nonumber-(p-2)\nabla^{\varphi}_{\frac{\partial}{\partial x}}|d\varphi|^{p-4}
   \left\langle\nabla^{\varphi}|\tau_{p}(\varphi)|^{q-2}\tau_{p}(\varphi),d\varphi\right\rangle d\varphi(\frac{\partial}{\partial x}).
\end{eqnarray}
By substituting formulas (\ref{eq28}) in (\ref{eq29}), we obtain
\begin{eqnarray*}
   \tau_{p,q}(\varphi)
   &=&-(p-1)(q-1)(ps-p-s)(pqs-pq-qs-s+1)\\
    & &(ps-p-s+1)^{q-1}s^{p-2+(p-1)(q-1)}x^{pqs-pq-qs-s}\frac{\partial}{\partial u}.
\end{eqnarray*}
\end{example}


\section{Liouville-type theorems for $(p,q)$-harmonic maps}

Liouville-type theorems play a central role in the study of harmonic and \(p\)-harmonic maps, often establishing rigidity results under geometric assumptions. A classical example is the result that any harmonic map from a compact Riemannian manifold into a Riemannian manifold with non-positive sectional curvature is constant map. The following theorem provides an affirmative answer in the case where the domain is compact and the target manifold has non-positive sectional curvature. It shows that under these classical geometric assumptions, the \((p,q)\)-harmonicity condition reduces to \(p\)-harmonicity.

\begin{theorem}\label{compact-case}
  Let $(M,g)$ be a compact orientable Riemannian manifold without boundary, and $(N,h)$ a Riemannian manifold with non-positive sectional curvature.
  Then, every $(p,q)$-harmonic map from $(M,g)$ to $(N,h)$ is $p$-harmonic.
\end{theorem}

\begin{proof}
  Let $\varphi:(M,g)\rightarrow(N,h)$ be a smooth $(p,q)$-harmonic map. We have
\begin{eqnarray*}
   0&=&-|\tau_{p}(\varphi)|^{q-2}|d\varphi|^{p-2}\operatorname{trace}R^{N}\big(\tau_{p}(\varphi),d\varphi\big)d\varphi\\
   &&- \operatorname{trace}\nabla^{\varphi}|d\varphi|^{p-2} \nabla^{\varphi}|\tau_{p}(\varphi)|^{q-2}\tau_{p}(\varphi)  \\
   && -(p-2)\operatorname{trace}\nabla^{\varphi}|d\varphi|^{p-4}
   \left\langle\nabla^{\varphi}|\tau_{p}(\varphi)|^{q-2}\tau_{p}(\varphi),d\varphi\right\rangle d\varphi. \end{eqnarray*}
Fix a point $x\in M$, and let $\{e_{i}\}$ be an orthonormal frame with respect to $g$ on $M$ such that $\nabla_{e_{i}}^{M}e_{j}=0$ at $x\in M$ for all $i,j=1, ...,m$. At $x$, we get\\

$|\tau_{p}(\varphi)|^{2(q-2)}|d\varphi|^{p-2}h\big(R^{N}(\tau_{p}(\varphi),d\varphi(e_{i}))d\varphi(e_{i}),\tau_{p}(\varphi)\big)$
\begin{eqnarray}\label{equa015}
   &=&- h\big(\nabla_{e_{i}}^{\varphi}|d\varphi|^{p-2} \nabla_{e_{i}}^{\varphi}|\tau_{p}(\varphi)|^{q-2}\tau_{p}(\varphi),|\tau_{p}(\varphi)|^{q-2}\tau_{p}(\varphi)\big) \\
   &&\nonumber -(p-2)h\big(\nabla_{e_{i}}^{\varphi}|d\varphi|^{p-4}
   \left\langle\nabla^{\varphi}|\tau_{p}(\varphi)|^{q-2}\tau_{p}(\varphi),d\varphi\right\rangle d\varphi(e_{i}),|\tau_{p}(\varphi)|^{q-2}\tau_{p}(\varphi)\big).
\end{eqnarray}
Let $\theta_{1},\theta_{2}\in\Gamma(T^{*}M)$ defined by
\begin{eqnarray*}
  \theta_{1}(X) &=& |d\varphi|^{p-2}h\big(\nabla_{X}^{\varphi}|\tau_{p}(\varphi)|^{q-2}\tau_{p}(\varphi),|\tau_{p}(\varphi)|^{q-2}\tau_{p}(\varphi)\big), \\
  \theta_{2}(X) &=&(p-2)|d\varphi|^{p-4}\left\langle\nabla^{\varphi}|\tau_{p}(\varphi)|^{q-2}\tau_{p}(\varphi),d\varphi\right\rangle
  h\big(d\varphi(X),|\tau_{p}(\varphi)|^{q-2}\tau_{p}(\varphi)\big),
\end{eqnarray*}
for all $X\in \Gamma(TM)$. Equation (\ref{equa015}) becomes\\

$|\tau_{p}(\varphi)|^{2(q-2)}|d\varphi|^{p-2}h\big(R^{N}(\tau_{p}(\varphi),d\varphi(e_{i}))d\varphi(e_{i}),\tau_{p}(\varphi)\big)$
\begin{eqnarray}\label{eq3.2}
   &=&\nonumber -\operatorname{div}\theta_{1}
   +|d\varphi|^{p-2}h\big( \nabla_{e_{i}}^{\varphi}|\tau_{p}(\varphi)|^{q-2}\tau_{p}(\varphi),
   \nabla_{e_{i}}^{\varphi}|\tau_{p}(\varphi)|^{q-2}\tau_{p}(\varphi)\big)-\operatorname{div}\theta_{2}\\
   &&\nonumber+(p-2)|d\varphi|^{p-4}
   \left\langle\nabla^{\varphi}|\tau_{p}(\varphi)|^{q-2}\tau_{p}(\varphi),d\varphi\right\rangle h\big(d\varphi(e_{i}),\nabla_{e_{i}}^{\varphi}|\tau_{p}(\varphi)|^{q-2}\tau_{p}(\varphi)\big).\\
\end{eqnarray}
From (\ref{eq3.2}) and assumption $\operatorname{Sect}^{N}\leq0$, we conclude that
\begin{eqnarray}\label{eq3.3}
   0&\geq&\nonumber -\operatorname{div}\theta_{1}
   +|d\varphi|^{p-2}\big|\nabla_{}^{\varphi}|\tau_{p}(\varphi)|^{q-2}\tau_{p}(\varphi)\big|^2
   -\operatorname{div}\theta_{2}\\
   &&\nonumber+(p-2)|d\varphi|^{p-4}
   \left\langle\nabla^{\varphi}|\tau_{p}(\varphi)|^{q-2}\tau_{p}(\varphi),d\varphi\right\rangle^2.
\end{eqnarray}
By using the Green Theorem (see \cite{baird}) and inequality (\ref{eq3.3}), we deduce
\begin{equation}\label{eq3.4}
\nabla_{X}^{\varphi}|\tau_{p}(\varphi)|^{q-2}\tau_{p}(\varphi)=0,\quad\forall X\in\Gamma(TM).
\end{equation}
Let $\theta_3(X)=h(|\tau_{p}(\varphi)|^{q-2}\tau_{p}(\varphi),|d\varphi|^{p-2}d\varphi(X))$. By using
(\ref{eq3.4}) and the definition of $\tau_{p}(\varphi)$, we find that
\begin{equation}\label{eq3.5}
    \operatorname{div}\theta_{3}=|\tau_{p}(\varphi)|^{q}.
\end{equation}
The Theorem \ref{compact-case} follows from (\ref{eq3.5}) and Green Theorem.
\end{proof}


For the case of the non-compact Riemannian manifold, we obtain the following result.

\begin{theorem}
  Let $(M,g)$ be a complete non-compact Riemannian manifold, and $(N,h)$ a Riemannian manifold with non-positive sectional curvature. We assume that
$$\int_{M}|d\varphi|^{p-2}|\tau_{p}(\varphi)|^{2(q-1)}v^g<\infty,
\quad\int_{M}|d\varphi|^{p-2}v^g=\infty.$$
Then, every $(p,q)$-harmonic map from $(M,g)$ to $(N,h)$ is $p$-harmonic.
\end{theorem}

\begin{proof}
  Let $\varphi:(M,g)\rightarrow(N,h)$ be a smooth $(p,q)$-biharmonic map, $\{e_{i}\}$ an orthonormal frame with respect to $g$ on $M$ such that $\nabla_{e_{i}}^{M}e_{j}$ at $x\in M$ for all $i,j=1, ...,m$, and $\rho$ a smooth function with compact support on $M$. Since the sectional curvature of $(N,h)$ is non-positive, from (\ref{equa015}) we obtain at $x$
\begin{eqnarray}\label{eq3.6}
   0&\geq&- h\big(\nabla_{e_{i}}^{\varphi}|d\varphi|^{p-2} \nabla_{e_{i}}^{\varphi}|\tau_{p}(\varphi)|^{q-2}\tau_{p}(\varphi),\rho^2|\tau_{p}(\varphi)|^{q-2}\tau_{p}(\varphi)\big) \\
   &&\nonumber -(p-2)h\big(\nabla_{e_{i}}^{\varphi}|d\varphi|^{p-4}
   \left\langle\nabla^{\varphi}|\tau_{p}(\varphi)|^{q-2}\tau_{p}(\varphi),d\varphi\right\rangle d\varphi(e_{i}),\rho^2|\tau_{p}(\varphi)|^{q-2}\tau_{p}(\varphi)\big).
\end{eqnarray}
Let $\beta_{1},\beta_{2}\in\Gamma(T^{*}M)$ defined by
\begin{eqnarray*}
\beta_{1}(X)&=&\rho^{2}|d\varphi|^{p-2}|\tau_{p}(\varphi)|^{q-2}h\big( \nabla_{X}^{\varphi}|\tau_{p}(\varphi)|^{q-2}\tau_{p}(\varphi),\tau_{p}(\varphi)\big),\\
\beta_{2}(X)&=&\rho^{2}|\tau_{p}(\varphi)|^{q-2}|d\varphi|^{p-4}\left\langle\nabla^{\varphi}|\tau_{p}(\varphi)|^{q-2}\tau_{p}(\varphi),
d\varphi\right\rangle h\big(d\varphi(X),\tau_{p}(\varphi)\big),
\end{eqnarray*}
for all $X\in \Gamma(TM)$. The last inequality is equivalent to the following
\begin{eqnarray}\label{equa017}
   0&\geq& -\text{div}\beta_{1}-(p-2)\text{div}\beta_{2} \\
   &&+  \rho^{2}|d\varphi|^{p-2}h\big( \nabla_{e_{i}}^{\varphi}|\tau_{p}(\varphi)|^{q-2}\tau_{p}(\varphi),\nabla_{e_{i}}^{\varphi}|\tau_{p}(\varphi)|^{q-2}\tau_{p}(\varphi)\big) \nonumber\\
   &&+2\rho e_{i}(\rho)|d\varphi|^{p-2}|\tau_{p}(\varphi)|^{q-2}h\big( \nabla_{e_{i}}^{\varphi}|\tau_{p}(\varphi)|^{q-2}\tau_{p}(\varphi),\tau_{p}(\varphi)\big) \nonumber \\
&& +(p-2)\rho^{2}|d\varphi|^{p-4}\left\langle\nabla^{\varphi}|\tau_{p}(\varphi)|^{q-2}\tau_{p}(\varphi),d\varphi\right\rangle h\big(d\varphi(e_{i}),
\nabla_{e_{i}}^{\varphi}|\tau_{p}(\varphi)|^{q-2}\tau_{p}(\varphi)\big) \nonumber\\
   &&+2(p-2)\rho e_{i}(\rho)|d\varphi|^{p-4}|\tau_{p}(\varphi)|^{q-2}
   \left\langle\nabla^{\varphi}|\tau_{p}(\varphi)|^{q-2}\tau_{p}(\varphi),d\varphi\right\rangle
   h\big(d\varphi(e_{i}),\tau_{p}(\varphi)\big).\nonumber
\end{eqnarray}
By using the Young's inequality, we have\\

$-2\rho e_{i}(\rho)|d\varphi|^{p-2}|\tau_{p}(\varphi)|^{q-2}h\big( \nabla_{e_{i}}^{\varphi}|\tau_{p}(\varphi)|^{q-2}\tau_{p}(\varphi),\tau_{p}(\varphi)\big)$
\begin{eqnarray}\label{eq3.8}
   &\leq&\frac{1}{2}\rho^{2}|d\varphi|^{p-2} \big|\nabla_{}^{\varphi}|\tau_{p}(\varphi)|^{q-2}\tau_{p}(\varphi)\big|^2
    +2|d\varphi|^{p-2}|\tau_{p}(\varphi)|^{2(q-1)}|\text{grad}\lambda|^2,\qquad
\end{eqnarray}
and the following\\

$-2\rho e_{i}(\rho)|d\varphi|^{p-4}|\tau_{p}(\varphi)|^{q-2}
   \left\langle\nabla^{\varphi}|\tau_{p}(\varphi)|^{q-2}\tau_{p}(\varphi),d\varphi\right\rangle
   h\big(d\varphi(e_{i}),\tau_{p}(\varphi)\big)$
\begin{eqnarray}\label{eq3.9}
   &\leq&\rho^{2}|d\varphi|^{p-4}\left\langle\nabla^{\varphi}|\tau_{p}(\varphi)|^{q-2}\tau_{p}(\varphi),d\varphi\right\rangle^2
         +|d\varphi|^{p-2}|\tau_{p}(\varphi)|^{2(q-1)}|\text{grad}\lambda|^2.\qquad
\end{eqnarray}
By substituting (\ref{eq3.8}) and (\ref{eq3.9}) in (\ref{equa017}), we find that
\begin{eqnarray}\label{eq3.10}
 p|d\varphi|^{p-2}|\tau_{p}(\varphi)|^{2(q-1)}|\text{grad}\lambda|^2
   &\geq& -\text{div}\beta_{1}-(p-2)\text{div}\beta_{2} \\
   & &\nonumber  +\frac{1}{2}\rho^{2}|d\varphi|^{p-2} \big|\nabla_{}^{\varphi}|\tau_{p}(\varphi)|^{q-2}\tau_{p}(\varphi)\big|^2.
\end{eqnarray}
From inequality (\ref{eq3.10}) and the divergence Theorem (see \cite{baird}), we get
\begin{eqnarray}\label{eq3.11}
 p\int_M|d\varphi|^{p-2}|\tau_{p}(\varphi)|^{2(q-1)}|\text{grad}\lambda|^2v^g
   &\geq&\nonumber \frac{1}{2}\int_M\rho^{2}|d\varphi|^{p-2} \big|\nabla_{}^{\varphi}|\tau_{p}(\varphi)|^{q-2}\tau_{p}(\varphi)\big|^2v^g.\\
\end{eqnarray}
Let $\rho=\rho_R:M\rightarrow [0,1]$ be a smooth cut-off function with $\rho=1$ on $B_{R}(x)$, $\rho= 0$ off $B_{2R}(x)$ and $|\text{grad}\rho|\leq \frac{2}{R}$. By using (\ref{eq3.11}), we obtain
\begin{eqnarray}\label{eq3.12}
 \frac{4p}{R^2}\int_{B_{2R}(x)}|d\varphi|^{p-2}|\tau_{p}(\varphi)|^{2(q-1)}v^g
   &\geq&\nonumber \frac{1}{2}\int_{B_{R}(x)}|d\varphi|^{p-2} \big|\nabla_{}^{\varphi}|\tau_{p}(\varphi)|^{q-2}\tau_{p}(\varphi)\big|^2v^g.\\
\end{eqnarray}
As $\int_{M}|d\varphi|^{p-2}|\tau_{p}(\varphi)|^{2(q-1)}v^g<\infty$ when $R\rightarrow \infty$, from (\ref{eq3.12}) we get
\begin{equation}\label{eq3.13}
    \nabla_{X}^{\varphi}|\tau_{p}(\varphi)|^{q-2}\tau_{p}(\varphi)=0,\quad\forall X\in\Gamma(TM).
\end{equation}
Equation (\ref{eq3.13}) implies that the function $|\tau_{p}(\varphi)|^{2(q-1)}$ is constant on $M$. By assumption $\int_{M}|d\varphi|^{p-2}v^g=\infty$, we conclude that $\tau_{p}(\varphi)=0$, that is the map $\varphi$ is $p$-harmonic.
\end{proof}

\end{document}